\documentclass[12pt]{article}

\usepackage{latexsym,amssymb,upref,amsmath,amsthm,amsfonts}
\usepackage {color}
\usepackage {pstricks, pst-node}
\usepackage {graphicx}
\usepackage {latexsym}
\usepackage [T1]{fontenc}
\usepackage [cp1250]{inputenc}
\usepackage{bbm}
\usepackage{bbold}

\newcommand{\achr}{\mathrm{achr}}
\newcommand{\ro}{\mathbb R}
\newcommand{\co}{\mathbb C}

\newcommand{\bp}{\noindent\textit{Proof. }}
\newcommand{\ep}{\blacksquare}

\newcommand{\exc}{\mathrm{exc}}
\newcommand{\frq}{\mathrm{frq}}
\newcommand{\cov}{\mathrm{cov}}
\newcommand{\Cov}{\mathrm{Cov}}

\newcommand{\cM}{\mathcal{M}}

\newtheorem{theorem}{Theorem}
\newtheorem{lemma}[theorem]{Lemma}
\newtheorem{proposition}[theorem]{Proposition}

\newtheorem{claim}{Claim}
\newtheorem{corollary}[theorem]{Corollary}

\title{\bf The achromatic number of $K_6\square K_q$\\equals $2q+3$ if $q\ge41$ is odd}

\author{Mirko Hor\v n\' ak\thanks{\noindent E-mail address: mirko.hornak$@$upjs.sk}\\
Institute of Mathematics, P.J. \v{S}af\'arik University\\
Jesenn\'a 5, 040\ 01 Ko\v{s}ice, Slovakia}

\date{}

\begin{document}
\maketitle

\begin{abstract}
Let $G$ be a graph and $C$ a finite set of colours. A vertex colouring $f:V(G)\to C$ is complete provided that for any two distinct colours $c_1,c_2\in C$ there is $v_1v_2\in E(G)$ such that $f(v_i)=c_i$, $i=1,2$. The achromatic number of $G$ is the maximum number $\achr(G)$ of colours in a proper complete vertex colouring of $G$. In the paper it is proved that if $q\ge41$ is an odd integer, then the achromatic number of the Cartesian product of $K_6$ and $K_q$ is $2q+3$.
\end{abstract}


\section{Introduction}

Let $G$ be a finite simple graph and $C$ a finite set of colours. A vertex colouring $f:V(G)\to C$ is \textit{complete} if for any pair of distinct colours $c_1,c_2\in C$ one can find in $G$ an edge $\{v_1,v_2\}$ (often shortened to $v_1v_2$) such that $f(v_i)=c_i$, $i=1,2$. The \textit{achromatic number} of $G$, denoted by $\achr(G)$, is the maximum cardinality of the colour set in a proper complete vertex colouring of $G$.

The concept was introduced quite a long ago in Harary, Hedetniemi and Prins~\cite{HaHePr}, where it was proved the following:

\begin{theorem}\label{ip}
If $G$ is a graph, and an integer $k$ satisfies $\chi(G)\le k\le\achr(G)$, there exists a proper complete vertex colouring of $G$ using $k$ colours. $\qed$
\end{theorem}

\noindent There are still only few graph classes $\mathcal G$ such that $\achr(G)$ is known for all $G\in\mathcal G$. This is certainly related to the fact that determining the achromatic number is an NP-complete problem even for trees, see Cairnie and Edwards~\cite{CaE}. Two surveys are available on the topic, namely Edwards~\cite{E1} and Hughes and MacGillivray\cite{HuMG}; more generally, Chapter~12 in the book~\cite{ChaZ} by Chartrand and Zhang deals with complete vertex colourings. A comprehensive list of publications concerning the achromatic number is maintained by Edwards~\cite{E2}.

Some papers are devoted to the achromatic number of graphs constructed by graph operations. So, Hell and Miller~\cite{HeM} considered $\achr(G_1\times G_2)$, where $G_1\times G_2$ stands for the categorical product of graphs $G_1$ and $G_2$ (we follow here the notation by Imrich and Klav\v zar~\cite{IK}).

In this paper we are interested in the achromatic number of the Cartesian product $G_1\square G_2$ of graphs $G_1$ and $G_2$, the graph with $V(G_1\square G_2)=\{(v_1,v_2):v_i\in V(G_i),i=1,2\}$, in which $(v_1^1,v_2^1)(v_1^2,v_2^2)\in E(G_1\square G_2)$ if and only if there is $i\in\{1,2\}$ such that $v_i^1v_i^2\in E(G_i)$ and $v_{3-i}^1=v_{3-i}^2$. As observed by Chiang and Fu~\cite{ChiF1}, $\achr(G_1)=p$ and $\achr(G_2)=q$ implies $\achr(G_1\square G_2)\ge\achr(K_p\square K_q)$. This inequality motivates a special interest in the achromatic number of the Cartesian product of two complete graphs. From the obvious fact that $G_2\square G_1$ is isomorphic to $G_1\square G_2$ it is clear that when determining $\achr(K_p\square K_q)$ we may suppose without loss of generality $p\le q$. 

The problem of determining $\achr(K_p\square K_q)$ with $p\le4$ was solved in Hor\v n\'ak and Puntig\'an~\cite{HoPu} (for $p\le3$ the result was rediscovered in~\cite{ChiF1}) and that for $p=5$ in Hor\v n\'ak and P\v cola~\cite{HoPc1}, \cite{HoPc2}. In~\cite{ChiF2} Chiang and Fu proved that if $r$ is an odd projective plane order, then $\achr(K_{(r^2+r)/2}\square K_{(r^2+r)/2})=(r^3+r^2)/2$. (For $r=3$ the fact that $\achr(K_6\square K_6)=18$ was known already to Bouchet~\cite{B}.) 

Here we show that $\achr(K_6\square K_q)=2q+3$ if $q$ is an odd integer with $q\ge41$. This is the first of three papers devoted to completely solve the problem of finding $\achr(K_6\square K_q)$.

For $k,l\in\mathbb Z$ we denote \textit{integer intervals} by
\begin{equation*}
[k,l]=\{z\in\mathbb Z:k\le z\le l\},\qquad [k,\infty)=\{z\in\mathbb Z:k\le z\}. 
\end{equation*}
Further, with a set $A$ and $m\in[0,\infty)$ we use $\binom Am$ for the set of $m$-element subsets of $A$.

Now let $p,q\in[1,\infty)$. Under the assumption that $V(K_r)=[1,r]$, $r=p,q$, we have $V(K_p\square K_q)=[1,p]\times[1,q]$, while $E(K_p\square K_q)$ consists of edges $(i,j_1)(i,j_2)$ with $i\in[1,p]$, $j_1,j_2\in[1,q]$, $j_1\ne j_2$ and $(i_1,j)(i_2,j)$ with $i_1,i_2\in[1,p]$, $i_1\ne i_2$, $j\in[1,q]$. 

A vertex colouring $f:[1,p]\times[1,q]\to C$ of the graph $K_p\square K_q$ can be conveniently described using the $p\times q$ matrix $M=M(f)$ whose entry in the $i$th row and the $j$th column is $(M)_{i,j}=f(i,j)$. Such a colouring is proper if any row of $M$ consists of $q$ distinct entries and any column of $M$ consists of $p$ distinct entries. Further, $f$ is complete provided that any pair $\{\alpha,\beta\}\in\binom C2$ is \textit{good} in $M$ in the following sense: there are  $(i_1,j_1),(i_2,j_2)\in[1,p]\times[1,q]$ such that $\{(M)_{i_1,j_1}, (M)_{i_2,j_2}\}=\{\alpha,\beta\}$ and either $i_1=i_2$, which we express by saying that the pair $\{\alpha,\beta\}$ is \textit{row-based} (in $M$), or $j_1=j_2$, \textit{i.e.}, the pair $\{\alpha,\beta\}$ is \textit{column-based} (in $M$). 

Let $\mathcal M(p,q,C)$ denote the set of $p\times q$ matrices $M$ with entries from $C$ such that all rows (columns) of $M$ have $q$ ($p$, respectively) distinct entries, and each pair $\{\alpha,\beta\}\in\binom C2$ is good in $M$. So, if $f:[1,p]\times[1,q]\to C$ is a proper complete vertex colouring of $K_p\square K_q$, then $M(f)\in\mathcal M(p,q,C)$. 

Conversely, if $M\in\mathcal M(p,q,C)$, then the mapping $f_M:[1,p]\times[1,q]\to C$ with $f_M(i,j)=(M)_{i,j}$ is a proper complete vertex colouring of $K_p\square K_q$. Thus, we have proved:

\begin{proposition}\label{general}
If $p,q\in[1,\infty)$ and $C$ is a finite set, then the following statements are equivalent:

$\mathrm(1)$ There is a proper complete vertex colouring of $K_p\square K_q$ using as colours elements of $C$.

$\mathrm(2)$ $\mathcal M(p,q,C)\ne\emptyset$.\quad$\qed$
\end{proposition}

We have another evident result:

\begin{proposition}\label{wlog}
If $p,q\in[1,\infty)$, $C,D$ are finite sets, $M\in\mathcal M(p,q,C)$, mappings $\rho:[1,p]\to[1,p]$, $\sigma:[1,q]\to[1,q]$, $\pi:C\to D$ are bijections, and $M_{\rho,\sigma}$, $M_{\pi}$ are $p\times q$ matrices defined by $(M_{\rho,\sigma})_{i,j}=(M)_{\rho(i),\sigma(j)}$ and $(M_{\pi})_{i,j}=\pi((M)_{i,j})$, then $M_{\rho,\sigma}\in\mathcal M(p,q,C)$ and $M_{\pi}\in\mathcal M(p,q,D)$.\quad$\qed$
\end{proposition}

Let $M\in\mathcal M(p,q,C)$. The \textit{frequency} of a colour $\gamma\in C$ is the number $\frq(\gamma)$ of appearances of $\gamma$ in $M$, and the \textit{frequency} of $M$, denoted $\frq(M)$, is the minimum of frequencies of colours in $C$. A colour of frequency $l$ is an \textit{$l$-colour}. $C_l$ is the set of $l$-colours, $c_l=|C_l|$, and $C_{l+}$ is the set of colours of frequency at least $l$, $c_{l+}=|C_{l+}|$. We denote by $\ro(i)$ the set $\{(M)_{i,j}:j\in[1,q]\}$ of colours in the $i$th row of $M$ and by $\co(j)$ the set $\{(M)_{i,j}:i\in[1,p]\}$ of colours in the $j$th column of $M$. Further, for $k\in\{l,l+\}$ let
\begin{alignat*}{2}
\ro_k(i)&=C_k\cap\ro(i),\qquad &r_k(i)&=|\ro_k(i)|,\\
\co_k(j)&=C_k\cap\co(j), &c_k(j)&=|\co_k(j)|.\\
\end{alignat*}
\vskip-6mm
\noindent If $A\subseteq[1,p]$, $|A|\ge2$, then
\[
\ro(A)=\bigcap_{l\in A}(C_{|A|}\cap\ro(l)),\qquad r(A)=|\ro(A)|
\]
($\ro(A)$ is the set of $|A|$-colours appearing in all rows of $M$ numbered by elements of $A$). Provided that $A=\{i,j\},\{i,j,k\},\{i,j,k,l\}$, instead of $\ro(A)$ we write $\ro(i,j)$, $\ro(i,j,k)$, $\ro(i,j,k,l)$, while $r(A)$ is simplified to $r(i,j)$, $r(i,j,k)$, $r(i,j,k,l)$, respectively. With $\{i,j\}\subseteq[1,p]$ and $\{m,n\}\subseteq[1,q]$ we set
\begin{alignat*}{2}
\ro_{3+}(i,j)&=C_{3+}\cap\ro(i)\cap\ro(j),\qquad &r_{3+}(i,j)&=|\ro_{3+}(i,j)|,\\
\co(m,n)&=C_2\cap\co(m)\cap\co(n), &c(m,n)&=|\co(m,n)|.
\end{alignat*}
If $B\subseteq[1,p]$ and $3\le|B|\le p-2$, then 
\begin{equation*}
\ro^*(B)=\bigcup_{l=2}^{|B|}\bigcup_{A\in\binom Bl}\ro(A),\qquad r^*(B)=|\ro^*(B)|.
\end{equation*}
For $\gamma\in C$ let 
\begin{equation*}
\mathbb R(\gamma)=\{i\in[1,p]:\gamma\in\mathbb R(i)\}.
\end{equation*}

With $S\subseteq[1,p]\times[1,q]$ we say that a colour $\gamma\in C$ \textit{occupies a position in} $S$ if there is $(i,j)\in S$ such that $(M)_{i,j}=\gamma$. If $\emptyset\ne A\subseteq C$, the \textit{set of columns covered by} $A$ is
\begin{equation*}
\Cov(A)=\{j\in[1,q]:\co(j)\cap A\ne\emptyset\}.
\end{equation*}
We define $\cov(A)=|\Cov(A)|$, and with $A\in\{\{\alpha\},\{\alpha,\beta\}\}$ we use a simplified notation $\Cov(\alpha)$, $\Cov(\alpha,\beta)$ and $\cov(\alpha)$, $\cov(\alpha,\beta)$ instead of $\Cov(A)$ and $\cov(A)$.

\section{Lower bound}

\begin{proposition}\label{lb}
If $q\in[7,\infty)$ and $q\equiv1\pmod2$, then $\achr(K_6\square K_q)\ge2q+3$.
\end{proposition}
\begin{proof}
Let $s=\frac{q-3}2$, and let $M$ be the $6\times q$ matrix below. We show that $M\in\cM(6,q,C)$, where $C=[1,9]\cup X_s\cup Y_s\cup Z_s\cup T_s$, $U_s=\{u_i:i\in[1,s]\}$ for $U\in\{X,Y,Z,T\}$, and the sets $[1,9]$, $X_s,Y_s,Z_s,T_s$ are pairwise disjoint.

\setcounter{MaxMatrixCols}{13}
\begin{equation*}
\begin{pmatrix}
1 &2 &3 &x_1 &x_2 &\dots &x_{s-1} &x_s &y_1 &y_2 &\dots &y_{s-1} &y_s\\
4 &5 &6 &x_s &x_1 &\dots &x_{s-2} &x_{s-1} &z_1 &z_2 &\dots &z_{s-1} &z_s\\
7 &8 &9 &t_1 &t_2 &\dots &t_{s-1} &t_s &x_1 &x_2 &\dots &x_{s-1} &x_s\\
3 &1 &2 &z_1 &z_2 &\dots &z_{s-1} &z_s &t_1 &t_2 &\dots &t_{s-1} &t_s\\
5 &6 &4 &t_s &t_1 &\dots &t_{s-2} &t_{s-1} &y_s &y_1 &\dots &y_{s-2} &y_{s-1}\\
8 &9 &7 &y_1 &y_2 &\dots &y_{s-1} &y_s &z_s &z_1 &\dots &z_{s-2} &z_{s-1}
\end{pmatrix}
\end{equation*}

Since $s\ge2$, because of our assumptions on the structure of $C$ it is clear that elements in lines (rows and columns) of $M$ are pairwise distinct. Thus it is sufficient to show that each pair $\{\alpha,\beta\}\in\binom C2$ is good in $M$. 

If $\alpha,\beta\in[1,9]$, then both $\alpha$ and $\beta$ appear twice in the columns $1,2,3$, hence the pair $\{\alpha,\beta\}$ is column-based.

If $\alpha\in C$ and $\beta\in X_s\cup Y_s\cup Z_s\cup T_s$, realise that $\ro(\alpha)\in\mathcal R_1\cup\mathcal R_2$ and $\ro(\beta)\in\mathcal R_2$, where $\mathcal R_1=\{\{1,4\},\{2,5\},\{3,6\}\}$ and $\mathcal R_2=\{[1,3],\{1,5,6\},\{2,4,6\},$ $[3,5]\}$. As $R\cap R_2\ne\emptyset$ for any $R\in\mathcal R_1\cup\mathcal R_2$ and any $R_2\in\mathcal R_2$, the pair $\{\alpha,\beta\}$ is row-based.

So, Proposition~\ref{general} yields $\achr(K_6\square K_q)\ge|C|=4s+9=2q+3$.
\end{proof}

\section{Auxiliary results}

If $M\in\mathcal M(p,q,C)$ and $\gamma\in C_l$, the \textit{excess} of the colour $\gamma$ is the number defined by
\[ \exc(\gamma)=l(p+q-l-1)-(|C|-1). \]

\begin{lemma}\label{bgen}
If $p,q\in[1,\infty)$, $C$ is a finite set, $M\in\mathcal M(p,q,C)$ and $\gamma\in C_l$, then the following hold:

$1.$ $l\le\min(p,q)$;

$2.$ $\exc(\gamma)\ge0$;

$3.$ $l=\frq(M)$ implies $|C|\le\lfloor\frac{pq}l\rfloor$.
\end{lemma}

\begin{proof}
1. The assumption $l>\min(p,q)$ would mean, by the pigeonhole principle, that the colouring $f_M$ is not proper.

2. Because of Proposition~\ref{wlog} we may suppose without loss of generality $(M)_{i,i}=\gamma$ for all $i\in[1,l]$. For simplicity we use (w) to indicate that it is just Proposition~\ref{wlog}, which enables us to restrict our attention to matrices with a special property.

The colouring $f_M$ is complete, hence each of $|C|-1$ colours in $C\setminus\{\gamma\}$ must occupy a position in the set $S=\{(i,j):(i\le l\lor j\le l)\land i\ne j\}$. Thus, $|S|=ql+(p-l)l-l\ge |C|-1$ and $\exc(\gamma)=|S|-(|C|-1)\ge0$.

3. If $\gamma'\in C\setminus\{\gamma\}$, then $\frq(\gamma')\ge\frq(\gamma)=l$. Therefore, the total number of entries of the matrix $M$ is $pq\ge l|C|$, and the desired inequality follows. 
\end{proof}

From the proof of Lemma~\ref{bgen}.2 we see that the excess of a colour $\gamma\in C$ is equal to the maximum number of entries (other than $\gamma$) that can be deleted from $M$ so that each pair $\{\gamma,\gamma'\}\in\binom C2$ is good even in the ``partial matrix'' corresponding to the involved restriction of $f_M$.

The \textit{excess} of a matrix $M\in\mathcal M(p,q,C)$, denoted by $\exc(M)$, is the minimum of excesses of colours in $C$.

\begin{lemma}\label{exc}
If $p,q\in[1,\infty)$, $C$ is a finite set and $M\in\mathcal M(p,q,C)$, then $\exc(M)=\exc(\gamma)$, where $\gamma\in C$ and $\frq(\gamma)=\frq(M)$.
\end{lemma}

\begin{proof}
Let $m=\min(p,q)$. If $\alpha,\beta\in C$ and $k=\frq(\alpha)>\frq(\beta)$, then, by Lemma~\ref{bgen}.1, $k\le m$. As a consequence, $\exc(\alpha)\ge\exc(\beta)\ge\exc(f)$, since $\exc(\alpha)=g(k)$, where  $g(x)=x(p+q-x-1)-|C|+1$ is increasing in the interval $\langle1,\frac{p+q-1}2\rangle\supsetneqq\langle1,m-1\rangle$, and $p=q=m$ implies $g(m-1)=g(m)$.
\end{proof}

\begin{lemma}[see \cite{HoPu}, \cite{ChiF1}]\label{gen}
If $p,q\in[1,\infty)$ and $p\le q$, then 
\[ \achr(K_p\square K_q)\le\max(\min(l(p+q-l-1)+1,\lfloor pq/l\rfloor):l\in[1,p]).\ \qed \]
\end{lemma}

\begin{corollary}\label{K6}
If $q\in[7,\infty)$, then $\achr(K_6\square K_q)\le2q+7$.
\end{corollary} 

\begin{proof}
By Lemma~\ref{gen} with $p=6$ we obtain $\achr(K_6\square K_q)\le\max(q+5,2q+7,2q,\lfloor\frac{6q}4\rfloor,\lfloor\frac{6q}5\rfloor,q)=2q+7$. 
\end{proof}

\section{Properties of matrices in $\mathcal M(6,q,C)$}\label{prope}

Suppose we know that $\achr(K_6\square K_q)\ge2q+s-1$ for a pair $(q,s)$ 
with $q\in[7,\infty)$ and $s\in[1,\infty)$, and we want to prove that $\achr(K_6\square K_q)=2q+s-1$; clearly, because of Corollary~\ref{K6} it is sufficient to work with $s\le7$. Proceeding by the way of contradiction let $s'\in[s,7]$ satisfy $\achr(K_6\square K_q)=2q+s'$. By Theorem~\ref{ip} and Proposition~\ref{general} there is a $(2q+s)$-element set $C$ and a matrix $M\in\mathcal M(6,q,C)$. Our task will be accomplished by showing that the existence of $M$ leads to a contradiction. For that purpose we shall need properties of $M$. So in all claims of the present section we suppose that the notation corresponds to a matrix $M\in\mathcal M(6,q,C)$ with $q\in[7,\infty)$ and $|C|=2q+s\le2q+7$. We associate with $M$ an auxiliary graph $G$ with $V(G)=[1,6]$, in which $\{i,k\}\in E(G)$ if and only if $r(i,k)\ge1$.

\begin{claim}\label{+m}
The following statements are true:

$1.$ $c_1=0$;

$2.$ $c_l=0$ for $l\in[7,\infty)$;

$3.$ $c_2\ge3s$;

$4.$ $c_{3+}\le2q-2s$;

$5.$ $\sum_{i=3}^6 ic_i\le6q-6s$;

$6.$ $\frq(M)=2$;

$7.$ $\exc(M)=7-s$;

$8.$ $c_{4+}\le c_2-3s$;

$9.$ if $\{i,k\}\in\binom{[1,6]}2$, then $r(i,k)\le8-s$.
\end{claim}

\begin{proof}
1. If $c_1>0$, a 1-colour $\gamma\in C$ satisfies $\exc(\gamma)=q+4-(2q+s-1)=5-q-s<0$ in contradiction to Lemma~\ref{bgen}.2.

2. Use Lemma~\ref{bgen}.1.

3. By Claims~\ref{+m}.1 and \ref{+m}.2, counting the number of vertices of $K_6\square K_q$ we get $6q=\sum_{i=2}^6 ic_i$. Therefore, $3(2q+s)=3|C|=3(c_2+c_{3+})\le c_2+\sum_{i=2}^6 ic_i=c_2+6q$, which yields $c_2\ge3s$.

4. From $2(2q+s)+c_{3+}=2|C|+c_{3+}=2c_2+3c_{3+}\le\sum_{i=2}^6 ic_i=6q$ we obtain $c_{3+}\le2q-2s$.

5. The assertion of Claim~\ref{+m}.4 leads to $\sum_{i=3}^6 ic_i=\sum_{i=2}^6 ic_i-2c_2=6q-2c_2\le6q-6s$.

6. A consequence of Claims~\ref{+m}.1, \ref{+m}.3 and the assumption $s\in[1,7]$.

7. Since $\frq(M)=2$ (Claim~\ref{+m}.6), by Lemma~\ref{exc} we get $\exc(M)=2q+6-(2q+s-1)=7-s$.

8. We have $3(2q+s)-c_2+c_{4+}=3(c_2+c_3+c_{4+})-c_2+c_{4+}\le\sum_{i=2}^6 ic_i=6q$ and $c_{4+}\le c_2-3s$.

9. The inequality is trivial if $r(i,k)=0$. If $\gamma\in\ro(i,k)$, then each colour of $\ro(i,k)\setminus\{\gamma\}$ contributes one to the excess of $\gamma$, hence, by Claims~\ref{+m}.6 and \ref{+m}.7, $r(i,k)-1\le\exc(\gamma)=\exc(M)=7-s$ and $r(i,k)\le8-s$.
\end{proof}

\begin{claim}\label{exci}
If $\{i,k\}\in\binom{[1,6]}2$ and $r(i,k)\ge1$, then $r(i,k)+r_{3+}(i,k)\le8-s$.
\end{claim}
\bp With $\gamma\in\ro(i,k)$ each colour of $(\ro(i,k)\setminus\{\gamma\})\cup\ro_{3+}(i,k)$ makes a contribution of one to the excess of $\gamma$, hence $r(i,k)-1+r_{3+}(i,k)\le\exc(\gamma)=\exc(M)\le7-s$, and the claim follows. $\ep$

\begin{claim}\label{r2r*}
If $\{i,k\}\in\binom{[1,6]}2$, $r(i,k)\ge1$, $B\subseteq[1,6]$, $3\le|B|\le4$ and $B\cap\{i,k\}=\emptyset$, then $r(B)\le r^*(B)\le2|B|$.
\end{claim}
\bp Consider a colour $\gamma\in\ro(i,k)$ with $(M)_{i,j}=(M)_{k,l}=\gamma$ (where, of course, $j\ne l$). If $\beta\in\ro^*(B)$, there is $A\subseteq B$ with $|A|\ge2$ and $\beta\in\ro(A)$. The colour $\beta$ appears in none of the rows $i,k$, hence the pair $\{\beta,\gamma\}$ is good in $M$ only if $\beta$ occupies a position in the set $\bigcup_{m\in A}\{(m,j),$ $(m,l)\}\subseteq\bigcup_{m\in B}\{(m,j),(m,l)\}$. As a consequence, $r^*(B)=|\ro^*(B)|\le|\bigcup_{m\in B}\{(m,j),(m,l)\}|=2|B|$, and the inequality $r(B)\le r^*(B)$ follows from the fact that $\ro(B)\subseteq\ro^*(B)$. $\ep$

\begin{claim}\label{r3r3}
If $\{i,j,k,l,m,n\}=[1,6]$ and $r(i,j,k)\ge1$, then $r(l,m,n)\le9$.
\end{claim}
\bp There is nothing to prove if $\ro(l,m,n)=\emptyset$. Further, with $\alpha\in\ro(i,j,k)$ and $\beta\in\ro(l,m,n)$ the pair $\{\alpha,\beta\}$ is good in $M$ only if the colour $\beta$ occupies a position in the 9-element set $\{l,m,n\}\times\Cov(\alpha)$. $\ep$

\begin{claim}\label{max4+}
If $\Delta(G)\ge4$, then $q\le40-5s$.
\end{claim}
\bp Suppose (w) $\Delta(G)=\deg_G(1)\ge4$, and, moreover, let (w) the sequence $(r(1,k))_{k=2}^6$ be nondecreasing. Then $r(1,3)\ge1$, and there is $p\in[2,3]$ such that $r(1,p)\ge1$ and $r(1,k)=0$ for $k\in[2,p-1]$. Clearly, we have $\ro_{3+}(1)=\bigcup_{k=p}^{6}\ro_{3+}(1,k)$. Realise that, by Claim~\ref{+m}.1, $q=|\ro(1)|=r_2(1)+r_{3+}(1)$. The inequality $r(1,k)\ge1$ for $k\in[p,6]$ yields, by Claim~\ref{exci}, $r_{3+}(1,k)\le8-s-r(1,k)$; therefore,
\begin{align*}
q-r_2(1)&=r_{3+}(1)=\left|\bigcup_{k=p}^{6}\ro_{3+}(1,k)\right|\le\sum_{k=p}^6 r_{3+}(1,k)\\
&\le\sum_{k=p}^6[8-s-r(1,k)]=(7-p)(8-s)-\sum_{k=p}^6 r(1,k),
\end{align*}
and then, since $r_2(1)=\sum_{k=p}^6 r(1,k)$, we finish with  $q\le(7-p)(8-s)\le40-5s$. $\ep$

\begin{claim}\label{D3}
If $\Delta(G)=3$, $\{i,j,k,l,m,n\}=[1,6]$, $r(i,l)\ge1$, $r(j,l)\ge1$ and $r(k,l)\ge1$, then $r(l,m,n)\ge q+3s-24$.
\end{claim}
\bp We have $\Delta(G)=\deg_G(l)$, $\ro(l,m)=\ro(l,n)=\emptyset$, $\ro(l)=\ro_2(l)\cup\ro_{3+}(l)$, $\ro_2(l)=\ro(i,l)\cup\ro(j,l)\cup\ro(k,l)$ and $\ro_{3+}(l)=\ro(l,m,n)\cup\ro_{3+}(i,l)\cup\ro_{3+}(j,l)\cup\ro_{3+}(k,l)$. Proceeding similarly as in the proof of Claim~\ref{max4+} leads to $q-r_2(l)=r_{3+}(l)\le r(l,m,n)+[8-s-r(i,l)]+[8-s-r(j,l)]+[8-s-r(k,l)]=r(l,m,n)+3(8-s)-r_2(l)$, which yields the desired result. $\ep$

\section{Main theorem}

\begin{theorem}\label{main}
If $q\in[41,\infty)$ and $q\equiv1\pmod2$, then $\achr(K_6\square K_q)=2q+3$.
\end{theorem}

\begin{proof}
We proceed by the way of contradiction. As mentioned in the beginning of Section~\ref{prope}, we have to show that the existence of a matrix $M\in\mathcal M(6,q,C)$, where $C$ is a $(2q+4)$-element set of colours, leads to a contradiction. First notice that, by Claim~\ref{+m}, all colours of $C$ are of frequency $l\in[2,6]$, $c_2\ge12$, $c_{3+}\le2q-8$, $\sum_{i=3}^{6}ic_i\le6q-24$, $\frq(M)=2$, $\exc(M)=3$, and $\{i,k\}\in\binom{[1,6]}{2}$ implies $r(i,k)\le4$.

Since $q\ge41$, from Claim~\ref{max4+} we know that $\Delta(G)\le3$ for the auxiliary graph $G$. Besides that, $\deg_G(i)=d$ implies $r_2(i)\le4d$ for $i\in[1,6]$. 

\begin{claim}\label{Dle2}
$\Delta(G)\le2$.
\end{claim}
\bp If $\Delta(G)=3$, (w) $\deg_G(1)=3$, $r(1,4)\ge1$, $r(1,5)\ge1$ and $r(1,6)\ge1$. By Claim~\ref{D3} we have $r(1,2,3)\ge q-12\ge29$, and so Claim~\ref{r2r*} yields $r(4,5)=r(4,6)=r(5,6)=0$ (if $r(i,k)\ge1$ for $\{i,k\}\in\binom{[4,6]}2$, then $29\le r(1,2,3)\le r^*([1,3])\le2|[1,3]|=6$, a contradiction). Moreover, $r(4,5,6)=0$, for otherwise, by Claim~\ref{r3r3}, $30\le r(1,2,3)+r(4,5,6)\le18$, a contradiction. 

There is no $i\in[4,6]$ with $\deg_G(i)=3$, because then, again by Claim~\ref{D3}, $r(4,5,6)\ge q-12\ge29$ in contradiction to Claim~\ref{r3r3}. So, $\deg_G(i)\le2$, $i=4,5,6$, $2c_2=\sum_{i=1}^6 r_2(i)\le3\cdot12+3\cdot8=60$ and $c_2\le30$. Since $r(1,i)\ge1$, Claim~\ref{r2r*} yields $r^*([2,6]\setminus\{i\})\le2|2,6]\setminus\{i\}|=8$, $i=4,5,6$, and then $\rho^*=\sum_{i=4}^6 r^*([2,6]\setminus\{i\})\le24$. 

It is easy to see that among summands of type $r(A)$ with $A\subseteq[2,6]$, $2\le|A|\le3$, that appear when counting $\rho^*$, each of $r(A)$ with $A\in\binom{[2,6]}2\setminus\{\{4,5\},\{4,6\},\{5,6\}\}$ appears at least twice, and each of $r(A)$ with $A\in\binom{[2,6]}3\setminus\{\{4,5,6\}\}$ (which is a set belonging to $C_3\setminus\ro_3(1)$) appears at least once. Because of $r(4,5)=r(4,6)=r(5,6)=r(4,5,6)=0$ this leads to $2\sum_{\{i,k\}\in\binom{[2,6]}2} r(i,k)+c_3-r_3(1)\le \rho^*\le24$, which, having in mind that $2\sum_{\{i,k\}\in\binom{[2,6]}2} r(i,k)=2c_2-2r_2(1)$, yields $2c_2-2r_2(1)+c_3-r_3(1)\le24$. Together with the inequality $r_2(1)+r_3(1)\le q$ then $c_2+c_3\le q+24+r_2(1)-c_2\le q+24$, $2q+4=|C|=c_2+c_3+c_{4+}\le q+24+c_{4+}$, and so $c_{4+}\ge q-20\ge21$ in contradiction to $c_{4+}\le c_2-12\le18$. $\ep$
\vskip2mm

By Claim~\ref{Dle2} each component of $G$ is either a path or a cycle. 

\begin{claim}\label{noK2}
No component of the graph $G$ is $K_2$.
\end{claim}
\bp Let (w) $G$ have a component $K_2$ with vertex set $[1,2]$. Then $r(1,2)\in[1,4]$ and $r(i,k)=0$ for $(i,k)\in[1,2]\times[3,6]$. Further, (w) $\Cov(\ro(1,2))=[1,n]$ with $n\in[2,8]$.

If $r(1,2)\in[1,3]$, Claim~\ref{r2r*} yields $\rho=\sum_{\{i,k\}\in\binom{[3,6]}2}r(i,k)\le r^*([3,6])\le8$; then $c_2=r(1,2)+\rho\le11$, a contradiction.

If $r(1,2)=4$, then $n\ge4$ and $\rho=|C_2\setminus\ro(1,2)|=8$. In the case $n\in[5,8]$ there is $j\in[1,n]$ such that $\co(j)$ contains at most $\lfloor\frac{2\cdot8}n\rfloor\le3$ colours of $C_2\setminus\ro(1,2)$. Then, however, for a colour $\gamma\in\ro(1,2)\cap\co(j)$ the number of colours $\delta\in C_2\setminus\ro(1,2)$, for which the pair $\{\gamma,\delta\}$ is good in $M$, is at most seven, a contradiction.

Therefore $n=4$, 2-colours occupy all positions in $[1,6]\times[1,4]$, $c_2=12$, $c_{4+}=0$, and all positions in the set $[1,6]\times[5,q]$ are occupied by 3-colours. Among other things this means that 
\begin{equation}\label{impin}
c_3=\frac{6(q-4)}3=2q-8\ge q+(41-8)=q+33,
\end{equation}
$r_2(i)=4$, $r_3(i)=q-4$ for $i\in[1,6]$ and $r_2(k)\in[1,3]$ for $k\in[3,6]$.

First, it is clear that $\Delta(G)\le2$. Indeed, if (w) $r(3,4)\ge1$, $r(3,5)\ge1$ and $r(3,6)\ge1$, then each 3-colour occupies a position in $[1,2]\times[5,q]$ and in $\{3\}\times[5,q]$ as well so that $c_3\le q-4<2q-8$, a contradiction.

Further, if $r(i,k)=4$ for $\{i,k\}\in\binom{[3,6]}{2}$, then $\deg_G(m)=1$, $m=3,4,5,6$. Consequently, the subgraph of $G$ induced by the vertex set $[3,6]$ is $d$-regular for some $d\in[1,2]$.

If $d=1$, then $G$ is isomorphic to $3K_2$ and (w) $r(i,i+1)=4$, $i=3,5$. Clearly, a set $\ro(i,j,k)$ with $\{i,j,k\}\in\binom{[1,6]}3$ can be nonempty only if $\{i,j,k\}\cap\{l,l+1\}\ne\emptyset$, $l=1,3,5$. As a consequence the assumption $\ro(i,j,k)\ne\emptyset$ with $i<j<k$ implies $(i,j,k)\in\{(1,3,5),(1,3,6),(1,4,5),(1,4,6),(2,3,5),$ $(2,3,6),(2,4,5),(2,4,6)\})$.

Suppose that $\{(i_m,j_m):m\in[1,4]\}=\{(3,5),(3,6),(4,5),(4,6)\}=\{(k_m,l_m):m\in[1,4]\}$ and $\{i_m,j_m,k_m,l_m\}=[3,6]$ for $m\in[1,4]$. Then
\begin{equation}\label{3K2}
c_3=\sum_{m=1}^4[r(1,i_m,j_m)+r(2,k_m,l_m)] .
\end{equation}
Further, for $m,n\in[1,4]$ the sets $\{i_m,j_m\}$ and $\{k_n,l_n\}$ are disjoint if and only if $m=n$. Put 
\begin{alignat*}{2}
T(1)&=\{m\in[1,4]:r(1,i_m,j_m)\ge1\},\qquad &t(1)&=|T(1)|,\\
T(2)&=\{m\in[1,4]:r(2,k_m,l_m)\ge1\}, &t(2)&=|T(2)|
\end{alignat*}
and $T=T(1)\cap T(2)$. Let 
\begin{equation*}
\sigma(P)=\sum_{p\in P}[r(1,i_p,j_p)+r(2,k_p,l_p)]
\end{equation*}
for $P\subseteq[1,4]$. Using Claim~\ref{r3r3} we see that $m\in T(1)$ implies $r(2,k_m,l_m)\le9$, while $m\in T(2)$ means that $r(1,i_m,j_m)\le9$. Therefore, with $m\in T$ we have $r(1,i_m,j_m)+r(2,k_m,l_m)\le9+9=18$, and so $\sigma(T)\le18|T|$. 

If there is $i\in[1,2]$ with $t(i)=4$, then $q-4=r_3(3-i)\le4\cdot9$ and $q\le40$, a contradiction.

If there is $i\in[1,2]$ with $t(i)=1$ and $r(i,i_m,j_m)\ge1$, then $r_3(i)=r(i,i_m,j_m)=r_3(i_m)=r_3(j_m)=q-4\ge37$, hence $r(3-i,k_m,l_m)=r_3(3-i)=r_3(k_m)=r_3(l_m)=q-4\ge37$, which contradicts Claim~\ref{r3r3}.

We are left with the situation $t(1),t(2)\in[2,3]$ (and $|T|\le\min(t(1),t(2))$). Suppose (w) $t(1)\ge t(2)$. 

If $|T|=3$, then $T(1)=T=T(2)$, $[1,4]\setminus T=\{p\}$ and $r(1,i_p,j_p)=r(2,k_p,l_p)=0$, hence $2q-8=c_3=\sigma([1,4])=\sigma(T)\le18\cdot3=54$ and $q\le31$, a contradiction.

If $|T|=2$, then $n$ out of four summands that sum up to $\sigma([1,4]\setminus T)$ are positive, $n\le2$. Moreover, $\sigma([1,4]\setminus T)\le q-4$. The inequality is obvious provided that $n\le1$, while if $n=2$, $a\in T(1)\setminus T(2)$ and $b\in T(2)\setminus T(1)$, then with $e\in\{i_a,j_a\}\cap\{k_b,l_b\}$ we have $\sigma([1,4]\setminus T)=r(1,i_a,j_a)+r(2,k_b,l_b)\le r_3(e)=q-4$. Thus $c_3=\sigma(T)+\sigma([1,4]\setminus T)\le18\cdot2+(q-4)=q+32$ in contradiction to (\ref{impin}).

For $|T|=1$ we get $t(2)=2$. With $t(1)=2$ we obtain, similarly as in the case $|T|=2$, $c_3\le18+(q-4)=q+14$. So, assume that $t(1)=3$, $T=\{t\}$ and $T(2)\setminus T(1)=\{p\}$ (note that $p\ne t$).

Suppose first that $\{k_p,l_p\}\ne\{i_t,j_t\}$. In such a case $|\{k_p,l_p\}\cap\{i_t,j_t\}|=1$; moreover, since $\{k_p,l_p\}\subseteq[3,6]=\{i_t,j_t,k_t,l_t\}$, we have $|\{k_p,l_p\}\cap\{k_t,l_t\}|=1$, and there is $g\in\{k,l\}$ such that $\{k_p,l_p\}\cap\{k_t,l_t\}=\{g_t\}$. Then five from among eight summands in (\ref{3K2}) are positive, namely $r(1,i_t,j_t)$, $r(2,k_t,l_t)$, $r(2,k_p,l_p)$ and $r(1,i_m,j_m)$ with $m\in[1,4]\setminus\{t,p\}$. Having in mind that $g_t\notin\{i_t,j_t\}$, $g_t\notin\{i_p,j_p\}$, and each element of $[3,6]$ is involved in exactly two of the ordered pairs $(3,5),(3,6),(4,5)$, $(4,6)$, we see that except for $r(1,i_t,j_t)$ all mentioned positive summands correspond to colours of $\ro_3(g_t)$. That is why $c_3\le r(1,i_t,j_t)+r_3(g_t)\le9+(q-4)=q+5$, a contradiction to (\ref{impin}) again.

On the other hand, if $\{k_p,l_p\}=\{i_t,j_t\}$, then $|\{i_p,j_p\}\cap\{i_t,j_t\}|=|\{i_p,j_p\}\cap\{k_p,l_p\}|=0$, hence for $m\in[1,4]\setminus\{t,p\}$ we have $|\{i_m,j_m\}\cap\{i_t,j_t\}|=1$, and so positive summands in (\ref{3K2}) are $r(1,i_t,j_t)$, $r(2,k_t,l_t)$, $r(2,k_p,l_p)=r(2,i_t,j_t)$, $r(1,g_t,k_t)$ and $r(1,h_t,l_t)$, where $\{g,h\}=\{i,j\}$. Then 
\begin{align*}
q-4&=r_3(2)=r(2,k_t,l_t)+r(2,i_t,j_t),\\
q-4&=r_3(k_t)=r(2,k_t,l_t)+r(1,g_t,k_t),\\
\end{align*}
which yields
\begin{equation*}
r(2,i_t,j_t)=r(1,g_t,k_t)=q-4-r(2,k_t,l_t)\ge (q-4)-9=q-13
\end{equation*}
so that
\begin{equation*}
q-4=r_3(g_t)=r(1,i_t,j_t)+r(1,g_t,k_t)+r(2,i_t,j_t)\ge1+2(q-13)
\end{equation*}
and $q\le21$, a contradiction.

If $|T|=0$, then $t(1)=t(2)=2$, and, by symmetry,
$T(1)=[1,2]$, $T(2)=[3,4]$. We have $\{i_1,j_1\}\cup\{i_2,j_2\}=[3,6]$, for otherwise there is $e\in[3,6]\setminus(\{i_1,j_1\}\cup\{i_2,j_2\})$, hence $e\in\{i_3,j_3\}\cap\{i_4,j_4\}$, $e\notin\{k_3,l_3\}\cup\{k_4,l_4\}$ and $r_3(e)=0\ne q-4$, a contradiction. Thus, by symmetry we may assume that $i_1=3<i_2=4$.

Therefore, (w) $(i_1,j_1)=(3,5)$ and $(i_2,j_2)=(4,6)$, which means that $r(m,n,p)$ with $\{m,n,p\}\in\binom{[1,6]}{3}$ and $m<n<p$ is positive if and only if $(m,n,p)\in\{(1,3,5),(1,4,6),(2,3,6),(2,4,5)\}$. We have $r_3(6)=r(1,4,6)+r(2,3,6)=q-4\equiv1\pmod2$, hence
\begin{equation}\label{par}
r(1,4,6)\not\equiv r(2,3,6)\pmod2.
\end{equation}
Similarly, from $r_3(1)=q-4=r_3(3)$ it follows that $r(1,3,5)\not\equiv r(1,4,6)\pmod2$ and $r(1,3,5)\not\equiv r(2,3,6)\pmod2$ so that $r(1,4,6)\equiv r(2,3,6)\pmod2$, which contradicts (\ref{par}).
 
If $d=2$, $G$ has a 4-cycle component, (w) $\{\{3,4\},\{4,5\},\{5,6\},\{6,3\}\}\subseteq E(G)$. The completeness of $f_M$ implies $C_3\subseteq\ro(1,3,5)\cup\ro(1,4,6)\cup\ro(2,3,5)\cup\ro(2,4,6)$. So, 
$$c_3=[r(1,3,5)+r(2,4,6)]+[r(1,4,6)+r(2,3,5)].$$
Let 
$\bar{\mathcal T}=\{(1,3,5),(1,4,6),(2,3,5),(2,4,6)\}$,
\begin{equation*}
\mathcal T=\{(i,j,k)\in\bar{\mathcal T}:r(i,j,k)\ge1\}
\end{equation*}  
and $t=|\mathcal T|$. From $r_3(i)=q-4>0$ for $i\in[1,6]$ it follows that $t\ge2$.
If $(i,j,k),(l,m,n)\in\mathcal T$, $(i,j,k)\ne(l,m,n)$, then either $r(i,j,k)+r(l,m,n)\le9+9=18$ (by Claim~\ref{r3r3}, if $\{i,j,k\}\cap\{l,m,n\}=\emptyset$) or $r(i,j,k)+r(l,m,n)\le r_3(p)=q-4$ (if $p\in\{i,j,k\}\cap\{l,m,n\}$). As a consequence then $c_3\le b_t$, where $b_2=q-4$, $b_3=18+(q-4)=q+14$ and $b_4=18+18=36$. Thus $c_3\le\max(b_2,b_3,b_4)=q+14$, which contradicts (\ref{impin}). $\ep$

\begin{claim}\label{noK3}
No component of the graph $G$ is $K_3$.
\end{claim}
\bp Let (w) $G$ have the component $K_3$ with the vertex set $[1,3]$.

If $G=K_3\cup3K_1$, then $r(1,2)=r(1,3)=r(2,3)=4$, $c_2=12$ and $c_{4+}=0$. From Claim~\ref{exci} it follows that $r_{3+}(1,3)=0=r_{3+}(2,3)$, hence $r(3,i,k)\ge1$ with $\{i,k\}\in\binom{[1,6]\setminus\{3\}}2$ implies $\{i,k\}\in\{\{4,5\},\{4,6\},\{5,6\}\}$. By Claim~\ref{r2r*} then $r_3(3)=r(3,4,5)+r(3,4,6)+r(3,5,6)\le r^*([3,6])\le8$, hence $q=r_2(3)+r_3(3)\le(4+4)+8=16$, a contradiction.

If $G$ has besides the above $K_3$ another nontrivial component (of order at least 2) and $r(i,k)\ge1$ with $\{i,k\}\in\binom{[4,6]}2$, then, by Claim~\ref{r2r*}, $r(1,2)+r(1,3)+r(2,3)\le r^*([1,3])\le6$ and $r(4,5)+r(4,6)+r(5,6)\le r^*([4,6])\le6$, hence $12\le c_2\le6+6$, $c_2=12$, $r(1,2)+r(1,3)+r(2,3)=r(4,5)+r(4,6)+r(5,6)=6$, $c_{4+}=0$ and $c_3=2q-8$. It is easy to see that in such a case 2-colours fill in four columns of $M$ and $G=2K_3$. Consequently, $\ro(i,j,k)$ can be nonempty only if $\{i,j,k\}\in\{[1,3],[4,6]\}$. Since $41\le q=r_2(i)+r_3(i)=4+r_3(i)$ for $i\in[1,6]$, we have $r(1,2,3)\ge37$ and $r(4,5,6)\ge37$, which contradicts Claim~\ref{r3r3}. $\ep$

\begin{claim}\label{noPn}
No component of the graph $G$ is a path of order at least $3$.
\end{claim}
\bp Suppose that $G$ has a path component $P$ of order at least $3$.

If $G$ has besides $P$ another nontrivial component (of order at least $2$), then $G=P\cup P'$, where, by Claims~\ref{noK2} and \ref{noK3}, both $P$ and $P'$ are paths of order 3, (w) $V(P)=[1,3]$, $V(P')=[4,6]$ and $r(i,i+1)\ge1$ for $i=1,2,4,5$. Similarly as in the proof of Claim~\ref{noK3} it is easy to see that $r(1,2)+r(2,3)= r(4,5)+r(5,6)=6$. Since $r_2(2)=6$, there are colours $\alpha,\beta\in\ro(1,2)\cup\ro(2,3)$ such that $\Cov(\alpha)\cap\Cov(\beta)=\emptyset$. Then each colour of $\ro(4,5)\cup\ro(5,6)$ occupies a position in $[4,6]\times\Cov(\alpha)$ and a position in $[4,6]\times\Cov(\beta)$ as well so that $\Cov(\ro(4,5)\cup\ro(5,6))\subseteq\Cov(\alpha,\beta)$; this leads to a contradiction since $\cov(\ro(4,5)\cup\ro(5,6))\ge r_2(5)=6$ and $\cov(\alpha,\beta)=4$.

So, $P$ is the unique nontrivial component of $G$, (w) $V(P)=[1,p]$ and $E(P)=\{\{i,i+1\}:i\in[1,p-1]\}$. Since $12\le c_2=\sum_{i=1}^{p-1} r(i,i+1)\le4(p-1)$, we have $p\in[4,6]$. 

If $p=4$, then $r(i,i+1)=4$, $i=1,2,3$, $c_2=12$, $c_{4+}=0$ and
\begin{equation}\label{P4}
i\in[1,6]\Rightarrow r_3(i)=q-r_2(i)\ge q-8\ge33.
\end{equation}
It is easy to see that $\cov(\ro(1,2)\cup\ro(3,4))=4$, (w) $\Cov(\ro(1,2)\cup\ro(3,4))=[1,4]$. Further, colours of $\ro(1,2)\cup\ro(3,4)$ occupy all positions in $[1,4]\times[1,4]$, hence (w) $\Cov(\ro(2,3))=[5,n]$, where $n\in[8,12]$.

By Claim~\ref{exci} we know that $r(i,j,k)=0$ if there is $l\in[1,3]$ such that $\{l,l+1\}\subseteq\{i,j,k\}$. 

Suppose that $\gamma\in\ro(1,5,6)$. Since all pairs $\{\gamma,\delta\}$ with $\delta\in\ro(3,4)$ are good in $M$, two positions in $[5,6]\times[1,4]$ must be occupied by $\gamma$. Then, however, the number of pairs $\{\gamma,\varepsilon\}$ with $\varepsilon\in\ro(2,3)$ that are good in $M$ is at most two, a contradiction. So, $r(1,5,6)=0$, and an analogous reasoning shows that $r(4,5,6)=0$.

Claim~\ref{r2r*} yields $r(1,4,5)+r(1,4,6)\le r^*(\{1,4,5,6\})\le8$. A colour $\gamma\in\ro(2,5,6)$ as well as a colour $\delta\in\ro(3,5,6)$ occupies two positions in $[5,6]\times[1,4]$ (each pair $\{\gamma,\varepsilon\}$ with $\varepsilon\in\ro(3,4)$ and each pair $\{\delta,\zeta\}$ with $\zeta\in\ro(1,2)$ is good in $M$). Thus $r(2,5,6)+r(3,5,6)\le4$ and
\begin{equation}\label{sum1}
r(1,4,5)+r(1,4,6)+r(2,5,6)+r(3,5,6)\le12.
\end{equation}

The set of remaining triples $(i,j,k)\in[1,6]^3$ with $i<j<k$ such that $r(i,j,k)$ can be positive is $\mathcal T=\{(1,3,5),(1,3,6),(2,4,5),(2,4,6)\}$. Suppose that $r(i,j,k)\ge1$ with $(i,j,k)\in\mathcal T$ if and only if $(i,j,k)\in\{(i_l,j_l,k_l):l\in[1,t]\}$. We show that there is $m\in[1,6]$ with $r_3(m)\le30$ in contradiction to (\ref{P4}).

If $t=4$, then, by Claim~\ref{r3r3}, $r(i_l,j_l,k_l)\le9$, $l=1,2,3,4$, and so, using (\ref{sum1}), $r_3(m)\le12+2\cdot9=30$ for (any) $m\in[1,6]$.
\textit{}
If $t=3$ and $r(i,j,k)=0$ for $(i,j,k)\in\mathcal T$, then $r_3(m)\le12+9=21$ for $m\in\{i,j,k\}$. The same upper bound applies for $m\in[1,6]$ if $t=2$ and $\{i_1,j_1,k_1\}\cap\{i_2,j_2,k_2\}=\emptyset$.

If $t=2$ and $\{i_1,j_1,k_1\}\cap\{i_2,j_2,k_2\}\ne\emptyset$, for $m\in[1,6]\setminus(\{i_1,j_1,k_1\}\cup\{i_2,j_2,k_2\})$ we obtain $r_3(m)\le12$. The same inequality is available for $m\in[1,6]\setminus\{i_1,j_1,k_1\}$ if $t=1$ and for $m\in[1,6]$ if $t=0$.

If $p=5$, then, by Claim~\ref{r2r*}, from $r(4,5)\ge1$ it follows that $r(1,2)+r(2,3)\le r^*([1,3])\le6$; similarly, $r(1,2)\ge1$ yields $r(3,4)+r(4,5)\le r^*([3,5])\le6$. Then $12\le c_2=\sum_{i=1}^4 r(i,i+1)\le6+6$, $c_2=12$ and $r(1,2)+r(2,3)=6=r(3,4)+r(4,5)$. If $(M)_{1,j}=\gamma\in\ro(1,2)$, then all positions in $[3,5]\times\Cov(\gamma)$ are occupied by six distinct colours of $\ro(3,4)\cup\ro(4,5)$, hence $(M)_{3,j}\in\ro(3,4)$ and $\delta=(M)_{5,j}\in\ro(4,5)$. Analogously, all positions in $[1,3]\times\Cov(\delta)$ are occupied by six distinct colours of $\ro(1,2)\cup\ro(2,3)$, which implies $(M)_{3,j}\in\ro(2,3)$, a contradiction.

If $p=6$, then, by Claim~\ref{r2r*}, $r^*([1,6]\setminus[l,l+1])\le8$ for $l\in[1,5]$, hence
\begin{equation}\label{r*}
\rho^*=\sum_{l=1}^5 r^*([1,6]\setminus[l,l+1])\le40.
\end{equation}
It is easy to see that in the sum $\rho^*$ each of the summands $r(i,i+1)$ with $i\in[1,5]$ appears in the expression of $r^*([1,6]\setminus[l,l+1])$ for 
at least two $l$'s, while each of the summands $r(i,j,k)$ satisfying $\{i,j,k\}\in\binom{[1,6]}3\setminus\{\{1,3,5\},\{2,3,5\},$ $\{2,4,5\},\{2,4,6\}\}$ and $i<j<k$ appears for at least one $l$. Since $c_2=\sum_{l=1}^5 r(l,l+1)$, with
\begin{equation*}
\rho=r(1,3,5)+r(2,3,5)+r(2,4,5)+r(2,4,6)
\end{equation*}
the inequality (\ref{r*}) leads to 
\begin{equation}\label{40}
2c_2+c_3-\rho\le\rho^*\le40.
\end{equation}
Moreover, $r_2(l)+r_3(l)\le q$ implies
\begin{equation*}\label{2,5}
r_3(l)\le q-r_2(l)=q-[r(l-1,l)+r(l,l+1)]\le q-2,\ l=2,5,
\end{equation*}
and so, having in mind that, by Claim~\ref{r3r3}, $\min(r(1,3,5),r(2,4,6))\le9$,
\begin{align}\label{rho}
\rho&\le\min(r_3(2)+r(1,3,5),r_3(5)+r(2,4,6))\nonumber\\
&\le q-2+\min(r(1,3,5),r(2,4,6))\le q+7.
\end{align}
Since, by Claim~\ref{+m}.8, $c_{4+}-c_2\le-12$, using (\ref{40})--(\ref{rho}) we obtain
\begin{equation*}
2q+4=c_2+c_3+c_{4+}\le40+\rho+c_{4+}-c_2\le40+(q+7)-12=q+35,
\end{equation*}
and, finally, $q\le31$, a contradiction. $\ep$
\vskip2mm

With $l\in\mathbb Z$ and $m\in[2,\infty)$ we use $(l)_m$ to denote $n\in[1,m]$ satisfying $n\equiv l\pmod m$.

\begin{claim}\label{noC4}
No component of the graph $G$ is a $4$-cycle.
\end{claim}
\bp If $G$ has a 4-cycle component, (w) $r(i,(i+1)_4)\ge1$ for $i\in[1,4$]. Note that, by Claim~\ref{noK2}, $r(5,6)=0=r_2(5)=r_2(6)$. Let $i\in[1,4]$ and
$$P_i=\{((i+2)_4,5),((i+2)_4,6),(5,6)\}.$$
Since evidently
\begin{equation*}
\ro_{3+}(i)=\ro(i,(i+2)_4,5,6)\cup\bigcup_{j\in\{(i-1)_4,i\}}\ro_{3+}(j,(j+1)_4)\cup\bigcup_{(j,k)\in P_i}\ro(i,j,k),
\end{equation*}
by Claim~\ref{exci} we have
\begin{align*}
r_{3+}(i)&\le r(i,(i+2)_4,5,6)+\sum_{j\in\{(i-1)_4,i\}}r_{3+}(j,(j+1)_4)+\sum_{(j,k)\in P_i}r(i,j,k)\nonumber\\
&\le r(i,(i+2)_4,5,6)+\sum_{j\in\{(i-1)_4,i\}}[4-r(j,(j+1)_4)]+\sum_{(j,k)\in P_i}r(i,j,k).
\end{align*}
Therefore, realising that 
\begin{equation*}
\sum_{i=1}^4\sum_{j\in\{(i-1)_4,i\}}r(j,(j+1)_4)=\sum_{i=1}^4 r_2(i),
\end{equation*}
we obtain
\begin{equation}\label{+32}
\sum_{i=1}^4 r_{3+}(i)\le32+\sum_{i=1}^4\left[r(i,(i+2)_4,5,6)+\sum_{(j,k)\in P_i}r(i,j,k)\right]-\sum_{i=1}^4 r_2(i).
\end{equation}

On the right-hand side  of the inequality (\ref{+32}) there are among others (partial) summands $r(A)$ with $A\in\binom{[1,6]}3\cup\binom{[1,6]}4$, each such summand appears there with the frequency 0, 1 or 2, and the frequency is 2 if and only if $A\in\{\{1,3,5\},\{1,3,6\},\{2,4,5\},\{2,4,6\},\{1,3,5,6\},\{2,4,5,6\}\}$. Thus, with
\begin{align*}
\rho_3&=r(1,3,5)+r(1,3,6)+r(2,4,5)+r(2,4,6),\\
\rho_4&=r(1,3,5,6)+r(2,4,5,6)
\end{align*}
the inequality (\ref{+32}) leads to
\begin{equation}\label{4q}
4q=\sum_{i=1}^4[r_2(i)+r_{3+}(i)]\le32+\rho_3+\rho_4+c_3+c_4.
\end{equation}

Let us show that
\begin{equation}\label{rhosum}
\rho_3\le q+10.
\end{equation}
To see it let $a$ be the number of sets $A$ belonging to
$$\mathcal A=\{\{1,3,5\},\{1,3,6\},\{2,4,5\},\{2,4,6\}\}$$
with $r(A)\ge1$. 

If $a=4$, by Claim~\ref{r3r3} we have $\rho_3\le4\cdot9=36\le q+10$.

In the case $a=3$ there is $i\in[1,4]$ such that $\rho_3\le2\cdot9+r_3(i)$. Evidently, $r_3(i)\le q-r_2(i)=q-[r((i-1)_4,i)+r(i,(i+1)_4)]\le q-(4+4)=q-8$, hence $\rho_3\le q+10$.

If $a=2$, let the positive summands of $\rho_3$ be $r(i,j,k)$ and $r(l,m,n)$, $\{i,j,k\}\ne\{l,m,n\}$. If $\{i,j,k\}\cap\{l,m,n\}=\emptyset$, then $\rho_3\le2\cdot9\le q+10$, and otherwise, with $p\in\{i,j,k\}\cap\{l,m,n\}$,  we have $\rho_3\le r_3(p)\le q$.

If $a\in[0,1]$, then $\rho_3\le q-2$, since $r(i,j,k)\ge1$ with $\{i,j,k\}\in\mathcal A$ and $i<j<k$ implies $i\in[1,2]$, while $r_2(i)\ge1+1$.

Now realise that, by Claim~\ref{+m}.8, $\rho_4+c_3+c_4\le c_3+2c_{4+}=|C|+(c_{4+}-c_2)\le(2q+4)-12=2q-8$, and so, using (\ref{4q}) and (\ref{rhosum}), $4q\le32+(q+10)+(2q-8)=3q+34$ and $q\le34$, a contradiction. $\ep$

\begin{claim}\label{noC5}
No component of the graph $G$ is a $5$-cycle.
\end{claim}
\bp If $G$ has a 5-cycle component, (w) $r(i,(i+1)_5)\ge1$ for $i\in[1,5]$. Similarly as in the proof of Claim~\ref{noC4} for $i\in[1,5]$ we get
\begin{equation*}
r_{3+}(i)\le r((i-2)_5,i,(i+2)_5,6)+\sum_{j\in\{(i-1)_5,i\}}r_{3+}(j,(j+1)_5)+\sum_{(j,k)\in P_i}r(i,j,k),
\end{equation*}
this time with 
$$P_i=\{((i-2)_5,6),((i+2)_5,6),((i-2)_5,(i+2)_5)\},$$
which yields
\begin{equation}\label{C5-}
5q=\sum_{i=1}^5[r_2(i)+r_{3+}(i)]\le40+\rho_3+c_3+c_4,
\end{equation}
where 
\begin{equation}\label{rho3}
\rho_3=r(1,3,6)+r(1,4,6)+r(2,4,6)+r(2,5,6)+r(3,5,6)\le r_3(6)\le q.
\end{equation}
Moreover, $c_3+c_4\le2q+4-c_2\le2q-8$, and so, using (\ref{C5-}) and (\ref{rho3}), $5q\le40+q+(2q-8)=3q+32$ and $q\le16$, a contradiction. $\ep$

\begin{claim}\label{noC6}
The graph $G$ is not a $6$-cycle.
\end{claim}
\bp If $G$ is a 6-cycle, (w) $r(i,(i+1)_6)\ge1$ for $i\in[1,6]$. In this case $r_{3+}(i)$ is upper bounded by 
\begin{equation*}
r((i-2)_6,i,(i+2)_6,(i+3)_6)+\sum_{j\in\{(i-1)_6,i\}}[4-r(j,(j+1)_6)]+\sum_{(j,k)\in P_i}r(i,j,k)
\end{equation*}
with 
$$P_i=\{((i-2)_6,(i+2)_6),((i-2)_6,(i+3)_6),((i+2)_6,(i+3)_6)\},$$
and one can see that 
\begin{equation}\label{C6-}
6q\le48+\rho_3+c_3+c_4,
\end{equation}
where $\rho_3=2r(1,3,5)+2r(2,4,6)$. We can bound $\rho_3$ from above by $2q-4$. Indeed, if both $r(1,3,5)$ and $r(2,4,6)$ are positive, then Claim~\ref{r3r3} yields $\rho_3\le18+18=36\le2q-4$. On the other hand, if $r(i,j,k)=0$ with $(i,j,k)\in\{(1,3,5),(2,4,6)\}$, then $\rho_3\le2r_3(i+1)\le2[q-r_2(i+1)]=2q-2[r_2(i,i+1)+r_2(i+1,i+2)]\le2q-4$. Therefore, similarly as in the proof of Claim~\ref{noC5}, from (\ref{C6-}) we obtain $6q\le48+(2q-4)+(2q-8)=4q+36$ and $q\le18$, a contradiction. $\ep$
\vskip2mm

Thus, by Claims~\ref{Dle2}--\ref{noC6}, we conclude that $G=6K_1$ and $c_2=0$, which contradicts Claim~\ref{+m}.3. Therefore, Theorem~\ref{main} is proved. \end{proof} 
\vskip3mm

\noindent{\Large{\bf Acknowledgements}}
\vskip1mm

\noindent This work was supported by the Slovak Research and Development Agency under the contract APVV-19-0153.

\end{document}